\documentclass[a4paper,10pt]{article}

\usepackage{amsfonts}
\usepackage{amsthm}
\usepackage{amsmath}
\usepackage{amssymb}
\usepackage{mathtools} 

\usepackage[top=1in, bottom=1in, left=1in, right=1in]{geometry}
\usepackage{xcolor}

\usepackage{graphicx}
\usepackage{epstopdf}
\usepackage{subfigure}

\usepackage{authblk}

\theoremstyle{plain}
\newtheorem{theorem}{Theorem}[section]
\newtheorem{coroll}[theorem]{Corollary}
\newtheorem{lem}[theorem]{Lemma}
\newtheorem{propn}[theorem]{Proposition}

\theoremstyle{definition}
\newtheorem{defn}[theorem]{Definition}

\theoremstyle{remark}
\newtheorem{remark}[theorem]{Remark}
\newtheorem{example}[theorem]{Example}

\usepackage{xcolor}
\usepackage{soul}
\definecolor{afcol}{rgb}{1,0,0}

\providecommand{\Real}{\mathop{\rm Re}\nolimits}%
\providecommand{\keywords}[1]{\textbf{\textit{Keywords:}} #1}

\begin{document}

\title{On some analytic properties of tempered fractional calculus}

\date{}

\author[1]{Arran Fernandez\thanks{Email: \texttt{arran.fernandez@emu.edu.tr}}}
\author[2]{Ceren Ustao\u{g}lu\thanks{Email: \texttt{ceren.ustaoglu@final.edu.tr}}}

\affil[1]{{\small Department of Mathematics, Faculty of Arts and Sciences, Eastern Mediterranean University, Famagusta, Northern Cyprus, via Mersin 10, Turkey}}
\affil[2]{{\small Department of Computer Engineering, Faculty of Engineering, Final International University, Kyrenia, Northern Cyprus, via Mersin 10, Turkey}}

\maketitle

\begin{abstract}
We consider the integral and derivative operators of tempered fractional calculus, and examine their analytic properties. We discover connections with the classical Riemann--Liouville fractional calculus and demonstrate how the operators may be used to obtain special functions such as hypergeometric and Appell's functions. We also prove an analogue of Taylor's theorem and some integral inequalities to enrich the mathematical theory of tempered fractional calculus.
\end{abstract}


\section{Background} \label{sec:intro}

In fractional calculus, we study generalisations of the $n$ times repeated differentiation operator $\frac{\mathrm{d}^n}{\mathrm{d}x^n}$, and the corresponding $n$ times repeated integration operator, to the case where $n$ is permitted to be not only an integer but any real or complex number. This idea goes back to the 17th century \cite{dugowson,miller-ross}, and by the time of the 21st century many different fractional-calculus operators have been defined.

The most standard and commonly used of these operators are the \textbf{Riemann--Liouville} integrals and derivatives, defined respectively as follows:
\begin{alignat}{2}
\label{RLdef:int} \prescript{RL}{a}I^{\nu}_tf(t)&\coloneqq\frac{1}{\Gamma(\nu)}\int_a^t(t-u)^{\nu-1}f(u)\,\mathrm{d}u,\quad\quad&\mathrm{Re}(\nu)>0; \\
\label{RLdef:der} \prescript{RL}{a}D^{\nu}_tf(t)&\coloneqq\frac{\mathrm{d}^n}{\mathrm{d}t^n}\prescript{RL}{a}I^{n-\nu}_tf(t),\quad n\coloneqq\lfloor\mathrm{Re}(\nu)\rfloor+1,\quad\quad&\mathrm{Re}(\nu)\geq0.
\end{alignat}
We note that by using the notational convention $\prescript{RL}{a}D^{-\nu}_tf(t)=\prescript{RL}{a}I^{\nu}_tf(t)$, it is possible to obtain from the above formulae a definition of $\prescript{RL}{a}D^{\nu}_tf(t)$ for any $\nu\in\mathbb{C}$. This is called the Riemann--Liouville \textbf{differintegral} (covering both derivatives and integrals), and it is analytic in the complex variable $\nu$ \cite{samko-kilbas-marichev}. Discussion of function spaces and domains of definition for these operators may be found in \cite{miller-ross,samko-kilbas-marichev}, while applications are covered in, for example, \cite{baleanu-diethelm-scalas-trujillo,deng-zhang,hilfer}.

Many other ways of defining fractional derivatives and integrals have been proposed. Some of these can be proved equivalent to the Riemann--Liouville definition, while others are defined by slight variations of the Riemann--Liouville formula. For example, the Gr\"unwald-Letnikov fractional differintegral is defined by the limit of a convergent series, and it is equivalent to Riemann--Liouville \cite{oldham-spanier}; meanwhile, the Caputo fractional derivative is defined by switching the order of the operations in the right-hand side of \eqref{RLdef:der}, and it is \textit{not} equivalent to Riemann--Liouville \cite{caputo,baleanu-diethelm-scalas-trujillo}.

Other, more recently developed, models of fractional calculus appear to bear no relationship to Riemann--Liouville whatsoever, beyond the superficial similarity of being defined by an integral transform with a kernel function \cite{caputo-fabrizio,yang-srivastava-machado,atangana-baleanu,prabhakar,kilbas-saigo-saxena,kiymaz-cetinkaya-agarwal,ozarslan-ozergin,cetinkaya-kiymaz-agarwal-agarwal,ozarslan-ustaoglu1,ozarslan-ustaoglu2}. However, in these cases too, relationships with the Riemann--Liouville model have been discovered, often in the form of infinite convergent series \cite{baleanu-fernandez,fernandez-baleanu-srivastava}. Several authors have proposed criteria and classifications for the many different models of fractional calculus -- see for example \cite{fernandez-ozarslan-baleanu,hilfer-luchko,ortigueira-machado,ross} -- but as yet there is no consensus in the field on this issue.

In the current work, we shall focus on a model called \textbf{tempered fractional calculus} which has been a subject of investigation in recent years due to its applications in stochastic and dynamical systems. Much of the research into this model has focused on numerical methods and applications, but less work has been done on the pure analytic side. We undertake here to examine the operators in detail and prove various mathematical facts about them which are useful in establishing the theoretical foundations for this model of fractional calculus.

Our paper is organised as follows. In Section \ref{sec:analysis} we give a brief overview of the existing theory of this model of fractional calculus and then proceed to give a number of new results: constructing relationships between this model and the classical Riemann--Liouville one, demonstrating its relationship with various special functions, and computing its Mellin transform. In Section \ref{sec:Taylor} we prove an analogue of Taylor's theorem which is valid in this model of fractional calculus; in Section \ref{sec:intineq} we establish some integral inequalities; and in Section \ref{sec:concl} we conclude the paper.

\section{Analysis of tempered fractional calculus} \label{sec:analysis}

The integral transform which is now called a tempered fractional integral appears to have been first analysed in \cite{buschman}, but the associated model of fractional calculus has been described more explicitly in e.g. \cite{li-deng-zhao,meerschaert-sabzikar-chen}. Both these papers and the references therein contain a number of applications of tempered fractional calculus to stochastic processes, random walks, Brownian motion, diffusion, turbulence, etc. A recent paper \cite{jarad-abdeljawad-alzabut} from 2018 also re-discovered tempered fractional calculus by fractionalising the proportional derivatives defined in \cite{anderson-ulness}.

\begin{defn}[Tempered fractional integrals] \label{Def:Tint}
Let $[a,b]$ be a real interval and $\alpha,\beta\in\mathbb{C}$ be parameters with $\Real(\alpha)>0$, $\Real(\beta)\geq0$. The $(\alpha,\beta)$th \textbf{tempered fractional integral} of a function $f\in L^1[a,b]$ is defined by
\begin{equation}
\label{Tint:def}
\prescript{T}{a}I^{(\alpha,\beta)}_tf(t)=\frac{1}{\Gamma(\alpha)}\int_a^t(t-u)^{\alpha-1}e^{-\beta(t-u)}f(u)\,\mathrm{d}u,\quad\quad t\in[a,b].
\end{equation}
(In the original definitions \cite{li-deng-zhao,meerschaert-sabzikar-chen}, the parameters $\alpha,\beta$ were assumed to be real. But this condition is not mathematically required and so we omit it, passing directly to the complex case.)
\end{defn}

\begin{defn}[Tempered fractional derivatives] \label{Def:Tder}
Let $[a,b]$ be a real interval and $\alpha,\beta\in\mathbb{C}$ be parameters with $\Real(\alpha)\geq0$, $\Real(\beta)\geq0$. The $(\alpha,\beta)$th \textbf{tempered fractional derivative} of a function $f\in C^n[a,b]$ is defined by
\begin{equation}
\label{Tder:def}
\prescript{T}{a}D^{(\alpha,\beta)}_tf(t)=\left(\frac{\mathrm{d}}{\mathrm{d}t}+\beta\right)^n\left(\prescript{T}{a}I^{(n-\alpha,\beta)}_tf(t)\right),\quad\quad t\in[a,b],
\end{equation}
where the natural number $n$ is defined by
\begin{equation} \label{n:defn}
n\coloneqq\lfloor\mathrm{Re}(\alpha)\rfloor+1
\end{equation}
so that $n\in\mathbb{N}$ and $\mathrm{Re}(n-\alpha)>0$.
\end{defn}

For completeness, we include also the definitions of the so-called \textbf{generalised proportional fractional} integrals and derivatives, which were introduced recently in \cite{jarad-abdeljawad-alzabut} and which are essentially equivalent to the existing definitions provided above.

\begin{defn}[GPF integrals \cite{jarad-abdeljawad-alzabut}] \label{Def:GPFint}
Let $[a,b]$ be a real interval and $\rho\in\mathbb{R}$, $\alpha\in\mathbb{C}$ be parameters with $0<\rho\leq1$ and $\mathrm{Re}(\alpha)>0$. The $(\alpha,\rho)$th \textbf{left GPF integral} of a function $f\in L^1[a,b]$ is defined by
\begin{equation}
\label{GPFint:leftdef}
\prescript{GPF}{a}I^{(\alpha,\rho)}_tf(t)=\frac{1}{\rho^{\alpha}\Gamma(\alpha)}\int_a^t(t-u)^{\alpha-1}\exp\left(\frac{\rho-1}{\rho}(t-u)\right)f(u)\,\mathrm{d}u,\quad\quad t\in[a,b],
\end{equation}
and similarly the $(\alpha,\rho)$th \textbf{right GPF integral} of a function $f\in L^1[a,b]$ is defined by
\begin{equation}
\label{GPFint:rightdef}
\prescript{GPF}{t}I^{(\alpha,\rho)}_bf(t)=\frac{1}{\rho^{\alpha}\Gamma(\alpha)}\int_t^b(u-t)^{\alpha-1}\exp\left(\frac{\rho-1}{\rho}(u-t)\right)f(u)\,\mathrm{d}u,\quad\quad t\in[a,b].
\end{equation}
\end{defn}

\begin{defn}[GPF derivatives \cite{jarad-abdeljawad-alzabut}] \label{Def:GPFder}
Let $[a,b]$ be a real interval and $\rho\in\mathbb{R}$, $\alpha\in\mathbb{C}$ be parameters with $0<\rho\leq1$ and $\mathrm{Re}(\alpha)\geq0$. The $(\alpha,\rho)$th \textbf{left GPF derivative} of a function $f\in C^n[a,b]$ is defined by
\begin{equation}
\label{GPFder:leftdef}
\prescript{GPF}{a}D^{(\alpha,\rho)}_tf(t)=\left((1-\rho)+\rho\,\tfrac{\mathrm{d}}{\mathrm{d}t}\right)^n\left(\prescript{GPF}{a}I^{(n-\alpha,\rho)}_tf(t)\right),\quad\quad t\in[a,b],
\end{equation}
and similarly the $(\alpha,\rho)$th \textbf{right GPF derivative} of a function $f\in C^n[a,b]$ is defined by
\begin{equation}
\label{GPFder:rightdef}
\prescript{GPF}{t}D^{(\alpha,\rho)}_bf(t)=\left((1-\rho)-\rho\,\tfrac{\mathrm{d}}{\mathrm{d}t}\right)^n\left(\prescript{GPF}{t}I^{(n-\alpha,\rho)}_bf(t)\right),\quad\quad t\in[a,b],
\end{equation}
where in each case the natural number $n$ is defined by \eqref{n:defn}.
\end{defn}

\begin{remark}
It is clear from the definitions that the GPF differintegrals and tempered fractional differintegrals are essentially equivalent, via the following identities:
\begin{align*}
\prescript{GPF}{a}I^{(\alpha,\rho)}_tf(t)&=\frac{1}{\rho^{\alpha}}\prescript{T}{a}I^{\left(\alpha,\frac{1-\rho}{\rho}\right)}_tf(t), \\
\prescript{GPF}{a}D^{(\alpha,\rho)}_tf(t)&=\rho^{\alpha}\prescript{T}{a}D^{\left(\alpha,\frac{1-\rho}{\rho}\right)}_tf(t).
\end{align*}
We also note that the parameter $\alpha$ in both definitions can be interpreted as an order of differentiation: see in particular the semigroup property given by Proposition \ref{Prop:semigroup} below. The parameters $\beta$ and $\rho$ respectively are less straightforward to interpret; they arise from the exponential part of the kernel function in the integral.
\end{remark}

Before proceeding to our new analysis, we summarise some of the important existing results about tempered fractional calculus.






\begin{propn}[The semigroup property \cite{li-deng-zhao,jarad-abdeljawad-alzabut}] \label{Prop:semigroup}
Let $[a,b]$ be a real interval and $\alpha_1,\alpha_2,\beta\in\mathbb{C}$ be parameters with $\Real(\alpha_i)>0$, $\Real(\beta)>0$. Then for any $f\in L^1[a,b]$ and any $t\in[a,b]$, we have the following semigroup property for tempered fractional integrals:
\[\prescript{T}{a}I^{(\alpha_1,\beta)}_t\left(\prescript{T}{a}I^{(\alpha_2,\beta)}_tf(t)\right)=\prescript{T}{a}I^{(\alpha_1+\alpha_2,\beta)}_tf(t).\]
If $\mathrm{Re}(\alpha_1)>\mathrm{Re}(\alpha_2)$, then we also have a semigroup property for tempered fractional derivatives of tempered fractional integrals:
\[\prescript{T}{a}D^{(\alpha_2,\beta)}_t\left(\prescript{T}{a}I^{(\alpha_1,\beta)}_tf(t)\right)=\prescript{T}{a}I^{(\alpha_1-\alpha_2,\beta)}_tf(t).\]
However, the semigroup property is not universally valid in tempered fractional calculus, for example:
\[\prescript{T}{a}I^{(\alpha,\beta)}_t\left(\prescript{T}{a}D^{(\alpha,\beta)}_tf(t)\right)=f(t)-e^{-\beta(t-a)}\sum_{k=1}^n\frac{(t-a)^{\alpha-k}}{\Gamma(\alpha-k+1)}\lim_{t\rightarrow a^+}\left(\prescript{T}{a}I^{(k-\alpha,\beta)}_tf(t)\right),\]
where $n$ is the natural number defined by \eqref{n:defn}.
\end{propn}

\begin{propn}[\cite{li-deng-zhao}] \label{Prop:RLconn}
The tempered fractional integral and derivative can be written simply in terms of the Riemann--Liouville fractional integral and derivative as follows:
\begin{align*}
\prescript{T}{a}I^{(\alpha,\beta)}_tf(t)&=e^{-\beta t}\prescript{RL}{a}I^{\alpha}_t\left[e^{\beta t}f(t)\right], \\
\prescript{T}{a}D^{(\alpha,\beta)}_tf(t)&=e^{-\beta t}\prescript{RL}{a}D^{\alpha}_t\left[e^{\beta t}f(t)\right].
\end{align*}
As a consequence, we have the following identities for the generalised proportional fractional integral and derivative, which were not observed in \cite{jarad-abdeljawad-alzabut}:
\begin{align*}
\prescript{GPF}{a}I^{(\alpha,\rho)}_tf(t)&=\frac{1}{\rho^{\alpha}}\exp\left(-\tfrac{1-\rho}{\rho}\cdot t\right)\prescript{RL}{a}I^{\alpha}_t\left[\exp\left(\tfrac{1-\rho}{\rho}\cdot t\right)f(t)\right]; \\
\prescript{GPF}{a}D^{(\alpha,\rho)}_tf(t)&=\frac{1}{\rho^{\alpha}}\exp\left(-\tfrac{1-\rho}{\rho}\cdot t\right)\prescript{RL}{a}D^{\alpha}_t\left[\exp\left(\tfrac{1-\rho}{\rho}\cdot t\right)f(t)\right].
\end{align*}
\end{propn}

\subsection{Relationship with the Riemann--Liouville model}

An important trend in fractional calculus is the study of  interrelationships between different models and definitions. As discussed in Section \ref{sec:intro} above, many types of fractional calculus can be related directly or indirectly to the Riemann--Liouville operators \eqref{RLdef:int}--\eqref{RLdef:der}.

The recent work of \cite{fernandez-ozarslan-baleanu} proposes a general framework to cover many different fractional operators all of which can be written as infinite series of Riemann--Liouville differintegrals. It is proved therein that the GPF integral \eqref{GPFint:leftdef} can be considered as a special case of a generalised operator $\prescript{A}{a}I^{\alpha,\beta}_tf(t)$ which satisfies the series formula
\begin{equation}
\label{FOBseries}
\prescript{A}{a}I^{\alpha,\beta}_tf(t)=\sum_{m=0}^{\infty}c_m\Gamma(\alpha+m\beta)\prescript{RL}{a}I^{\alpha+m\beta}_tf(t),
\end{equation}
where $A(z)=\sum_{m=0}^{\infty}c_mz^m$ is an analytic function. In particular, \cite{fernandez-ozarslan-baleanu} shows that the GPF integral is the special case given by $\beta=1$ and
\[A(z)=\frac{1}{\rho^{\alpha}\Gamma(\alpha)}\exp\left(\frac{\rho-1}{\rho}z\right).\]
With a little more work, we have the following theorem.

\begin{theorem}[Series formula for tempered fractional differintegrals] \label{Thm:GPFseries}
With all notations as in Definitions \ref{Def:Tint} and \ref{Def:Tder} respectively, we have the following expressions for the tempered fractional integrals and derivatives as infinite convergent series of Riemann--Liouville differintegrals:
\begin{align}
\label{GPFint:series} \prescript{T}{a}I^{(\alpha,\beta)}_tf(t)=\sum_{m=0}^{\infty}\frac{(-\beta)^m\Gamma(\alpha+m)}{m!\Gamma(\alpha)}\prescript{RL}{a}I^{\alpha+m}_tf(t); \\
\label{GPFder:series} \prescript{T}{a}D^{(\alpha,\beta)}_tf(t)=\sum_{m=0}^{\infty}\frac{(\-\beta)^{-m}\Gamma(-\alpha+m)}{m!\Gamma(-\alpha)}\prescript{RL}{a}I^{-\alpha+m}_tf(t).
\end{align}
\end{theorem}

\begin{proof}
First we consider the tempered fractional integral. Using the notation above inspired by \cite{fernandez-ozarslan-baleanu}, we have
\[A(z)=\frac{1}{\Gamma(\alpha)}\exp\left(-\beta z\right)=\sum_{m=0}^{\infty}\frac{(-\beta)^mz^m}{\Gamma(\alpha)m!},\]
and therefore
\[c_m=\frac{(-\beta)^m}{m!\Gamma(\alpha)}\quad\forall m.\]
Using this and $\beta=1$ in the general series formula \eqref{FOBseries} yields
\[\prescript{T}{a}I^{(\alpha,\beta)}_tf(t)=\sum_{m=0}^{\infty}\frac{(-\beta)^m}{m!\Gamma(\alpha)}\Gamma(\alpha+m\cdot1)\prescript{RL}{a}I^{\alpha+m\cdot1}_tf(t)=\sum_{m=0}^{\infty}\frac{(-\beta)^m\Gamma(\alpha+m)}{m!\Gamma(\alpha)}\prescript{RL}{a}I^{\alpha+m}_tf(t),\]
and we obtain the required identity \eqref{GPFint:series}.

Alternatively, we could have proved \eqref{GPFint:series} directly by expanding the exponential kernel in the definition \eqref{GPFint:leftdef} as a power series and then using uniform convergence to swap the summation and integration. (This was also the method used in general to prove \eqref{FOBseries} in \cite{fernandez-ozarslan-baleanu}.)

Now we consider the tempered fractional derivative. By the definition \eqref{GPFder:leftdef}, the tempered fractional derivative $\prescript{T}{a}D^{(\alpha,\beta)}_tf(t)$ is obtained by applying the proportional derivative operator $\left(\frac{\mathrm{d}}{\mathrm{d}t}+\beta\right)$ repeatedly $n$ times to a tempered fractional integral $\prescript{T}{a}I^{(n-\alpha,\beta)}_tf(t)$, where $n$ is a natural number. We proceed by induction on $n$, noting first the following:
\begin{align*}
\left(\frac{\mathrm{d}}{\mathrm{d}t}+\beta\right)\left[\prescript{T}{a}I^{(\alpha,\beta)}_tf(t)\right]&=\left(\frac{\mathrm{d}}{\mathrm{d}t}+\beta\right)\left[\sum_{m=0}^{\infty}\frac{(-\beta)^m\Gamma(\alpha+m)}{m!\Gamma(\alpha)}\prescript{RL}{a}I^{\alpha+m}_tf(t)\right] \\
&\hspace{-2cm}=\sum_{m=0}^{\infty}\frac{(-\beta)^m\Gamma(\alpha+m)}{m!\Gamma(\alpha)}\frac{\mathrm{d}}{\mathrm{d}t}\prescript{RL}{a}I^{\alpha+m}_tf(t)+\sum_{m=0}^{\infty}\beta\frac{(-\beta)^m\Gamma(\alpha+m)}{m!\Gamma(\alpha)}\prescript{RL}{a}I^{\alpha+m}_tf(t) \\
&\hspace{-2cm}=\sum_{m=0}^{\infty}\frac{(-\beta)^m\Gamma(\alpha+m)}{m!\Gamma(\alpha)}\prescript{RL}{a}I^{\alpha-1+m}_tf(t)-\sum_{m=0}^{\infty}\frac{(-\beta)^{m+1}\Gamma(\alpha+m)}{m!\Gamma(\alpha)}\prescript{RL}{a}I^{\alpha+m}_tf(t) \\
&\hspace{-2cm}=\prescript{RL}{a}I^{\alpha-1}_tf(t)+\sum_{m=1}^{\infty}\frac{(-\beta)^m\Gamma(\alpha+m)}{m!\Gamma(\alpha)}\prescript{RL}{a}I^{\alpha-1+m}_tf(t)-\sum_{m=0}^{\infty}\frac{(-\beta)^{m+1}\Gamma(\alpha+m)}{m!\Gamma(\alpha)}\prescript{RL}{a}I^{\alpha+m}_tf(t) \\
&\hspace{-2cm}=\prescript{RL}{a}I^{\alpha-1}_tf(t)+\sum_{m=1}^{\infty}\frac{(-\beta)^m\Gamma(\alpha+m)}{m!\Gamma(\alpha)}\prescript{RL}{a}I^{\alpha-1+m}_tf(t)-\sum_{m=1}^{\infty}\frac{(-\beta)^{m}\Gamma(\alpha+m-1)}{(m-1)!\Gamma(\alpha)}\prescript{RL}{a}I^{\alpha+m-1}_tf(t) \\
&\hspace{-2cm}=\prescript{RL}{a}I^{\alpha-1}_tf(t)+\sum_{m=1}^{\infty}\frac{(-\beta)^m\Gamma(\alpha-1+m)}{m!\Gamma(\alpha)}\Big[(\alpha-1+m)-(m)\Big]\prescript{RL}{a}I^{\alpha-1+m}_tf(t) \\
&\hspace{-2cm}=\prescript{RL}{a}I^{\alpha-1}_tf(t)+\sum_{m=1}^{\infty}\frac{(-\beta)^m\Gamma(\alpha-1+m)}{m!\Gamma(\alpha-1)}\prescript{RL}{a}I^{\alpha-1+m}_tf(t) \\
&\hspace{-2cm}=\sum_{m=0}^{\infty}\frac{(-\beta)^m\Gamma(\alpha-1+m)}{m!\Gamma(\alpha-1)}\prescript{RL}{a}I^{\alpha-1+m}_tf(t).
\end{align*}
Note that the final series formula here is exactly the same as \eqref{GPFint:series} but with $\alpha$ replaced by $\alpha-1$. Applying the proportional derivative $\left(\frac{\mathrm{d}}{\mathrm{d}t}+\beta\right)$ multiple times, we have the following formula by induction on $n$:
\[\left(\frac{\mathrm{d}}{\mathrm{d}t}+\beta\right)^n\left[\prescript{T}{a}I^{(\alpha,\beta)}_tf(t)\right]=\sum_{m=0}^{\infty}\frac{(-\beta)^{m}\Gamma(\alpha-n+m)}{m!\Gamma(\alpha-n)}\prescript{RL}{a}I^{\alpha-n+m}_tf(t).\]
And now, by replacing $\alpha$ with $n-\alpha$, we find the series formula \eqref{GPFder:series} for tempered fractional derivatives:
\begin{align*}
\prescript{T}{a}D^{(\alpha,\beta)}_tf(t)&=\left(\frac{\mathrm{d}}{\mathrm{d}t}+\beta\right)^n\left[\prescript{T}{a}I^{(n-\alpha,\beta)}_tf(t)\right] \\
&=\sum_{m=0}^{\infty}\frac{(-\beta)^{m}\Gamma(n-\alpha-n+m)}{m!\Gamma(n-\alpha-n)}\prescript{RL}{a}I^{n-\alpha-n+m}_tf(t) \\
&=\sum_{m=0}^{\infty}\frac{(-\beta)^{m}\Gamma(-\alpha+m)}{m!\Gamma(-\alpha)}\prescript{RL}{a}I^{-\alpha+m}_tf(t).
\end{align*}
\end{proof}

\subsection{Examples and special functions}

Let us now consider some example functions $f(t)$ and how the tempered fractional operators act on them. In this way we shall derive some interesting relations between special functions and tempered fractional calculus.

\begin{example}
\label{example1}
Let $\boldsymbol{f(t)=(t-a)^{\lambda}}$, with $\mathrm{Re}(\lambda)>-1$. In this case, we have from Theorem \ref{Thm:GPFseries}:
\begin{align*}
\prescript{T}{a}I^{(\alpha,\beta)}_tf(t)&=\sum_{m=0}^{\infty}\frac{(-\beta)^m\Gamma(\alpha+m)}{m!\Gamma(\alpha)}\prescript{RL}{a}I^{\alpha+m}_t\left((t-a)^{\lambda}\right) \\
&=\sum_{m=0}^{\infty}\frac{(-\beta)^m\Gamma(\alpha+m)}{m!\Gamma(\alpha)}\cdot\frac{\Gamma(\lambda+1)}{\Gamma(\lambda+\alpha+m+1)}(t-a)^{\lambda+\alpha+m} \\
&=(t-a)^{\lambda+\alpha}\sum_{m=0}^{\infty}\frac{\Gamma(\alpha+m)}{m!\Gamma(\alpha)}\cdot\frac{\Gamma(\lambda+1)}{\Gamma(\lambda+\alpha+m+1)}(-\beta)^m(t-a)^{m} \\
&=(t-a)^{\lambda+\alpha}\cdot\frac{\Gamma(\lambda+1)}{\Gamma(\lambda+\alpha+1)}\sum_{m=0}^{\infty}\frac{\Gamma(\alpha+m)}{m!\Gamma(\alpha)}\cdot\frac{\Gamma(\lambda+\alpha+1)}{\Gamma(\lambda+\alpha+m+1)}\left(-\beta(t-a)\right)^m \\
&=(t-a)^{\lambda+\alpha}\cdot\frac{\Gamma(\lambda+1)}{\Gamma(\lambda+\alpha+1)}\prescript{}{1}F_1\left(\alpha;\lambda+\alpha+1;-\beta(t-a)\right),
\end{align*}
and similarly
\[\prescript{T}{a}D^{(\alpha,\beta)}_tf(t)=(t-a)^{\lambda-\alpha}\cdot\frac{\Gamma(\lambda+1)}{\Gamma(\lambda-\alpha+1)}\prescript{}{1}F_1\left(-\alpha;\lambda-\alpha+1;-\beta(t-a)\right),\]
where $\prescript{}{1}F_1$ is the confluent hypergeometric function. We note how a special function, albeit a well-known one, arises from applying the tempered fractional integral to just an elementary power function.
\end{example}

\begin{example}
\label{example2}
Let $\boldsymbol{f(t)=t^{\mu-1}(1-t)^{-\lambda}}$, with $\mathrm{Re}(\mu)>0$, and $a=0$. It is known \cite{miller-ross} that the Riemann--Liouville differintegral of this function $f(t)$ is given by
\[\prescript{RL}{0}I^{\nu}_t\left(t^{\mu-1}(1-t)^{-\lambda}\right)=\frac{\Gamma(\mu)}{\Gamma(\mu+\nu)}t^{\mu+\nu-1}\prescript{}{2}F_1(\mu,\lambda;\mu+\nu;t),\quad\quad\mathrm{Re}(\mu)>0,|t|<1,\]
where $\prescript{}{2}F_1$ is the Gauss hypergeometric function. Therefore, using Theorem \ref{Thm:GPFseries} again, we have:
\begin{align*}
\prescript{T}{0}I^{(\alpha,\beta)}_tf(t)&=\sum_{m=0}^{\infty}\frac{(-\beta)^m\Gamma(\alpha+m)}{m!\Gamma(\alpha)}\prescript{RL}{0}I^{\alpha+m}_t\left(t^{\mu-1}(1-t)^{-\lambda}\right) \\
&=\sum_{m=0}^{\infty}\frac{(-\beta)^m\Gamma(\alpha+m)}{m!\Gamma(\alpha)}\cdot\frac{\Gamma(\mu)}{\Gamma(\mu+\alpha+m)}t^{\mu+\alpha+m-1}\prescript{}{2}F_1(\mu,\lambda;\mu+\alpha+m;t) \\
&=t^{\mu+\alpha-1}\cdot\frac{\Gamma(\mu)}{\Gamma(\mu+\alpha)}\sum_{m=0}^{\infty}\frac{(-\beta)^m\Gamma(\alpha+m)}{m!\Gamma(\alpha)}\cdot\frac{\Gamma(\mu+\alpha)}{\Gamma(\mu+\alpha+m)}t^{m}\prescript{}{2}F_1(\mu,\lambda;\mu+\alpha+m;t) \\
&=t^{\mu+\alpha-1}\cdot\frac{\Gamma(\mu)}{\Gamma(\mu+\alpha)}\sum_{m=0}^{\infty}\frac{\Gamma(\alpha+m)\Gamma(\mu+\alpha)}{\Gamma(\alpha)\Gamma(\mu+\alpha+m)}\cdot\frac{(-\beta t)^m}{m!}\prescript{}{2}F_1(\mu,\lambda;\mu+\alpha+m;t).
\end{align*}
If the factor $\prescript{}{2}F_1(\mu,\lambda;\mu+\alpha+m;t)$ were removed from this expression, it would be precisely the series for the confluent hypergeometric function $\prescript{}{1}F_1(\alpha;\mu+\alpha;-\beta t)$. Thus, the tempered fractional integral of this $f(t)$ can be seen as a twisted convolution of the confluent hypergeometric function with the Gauss hypergeometric function.

Again, the result for the derivative $\prescript{T}{0}D^{(\alpha,\beta)}_tf(t)$ is exactly the same as for the integral but with $\alpha$ replaced by $-\alpha$.
\end{example}

\begin{example}
\label{example3}
Let $\boldsymbol{f(t)=t^{\mu-1}(1-at)^{-\lambda}(1-bt)^{-\beta}}$, with $\mathrm{Re}(\mu)>0$, and $a=0$. It is known \cite{picard} that the Riemann--Liouville differintegral of this function $f(t)$ is given by
\[\prescript{RL}{0}I^{\nu}_t\left(t^{\mu-1}(1-at)^{-\lambda}(1-bt)^{-\beta}\right)=\frac{\Gamma(\mu)}{\Gamma(\mu+\nu)}t^{\mu+\nu-1}F_1(\mu,\lambda,\beta,\mu+\nu;at,bt),\quad\quad\mathrm{Re}(\mu)>0,\mathrm{Re}(\nu)>0,\]
where $F_1$ is the first Appell function. Therefore, using Theorem \ref{Thm:GPFseries} again, we have:
\begin{align*}
\prescript{T}{0}I^{(\alpha,\beta)}_tf(t)&=\sum_{m=0}^{\infty}\frac{(-\beta)^m\Gamma(\alpha+m)}{m!\Gamma(\alpha)}\prescript{RL}{0}I^{\alpha+m}_t\left(t^{\mu-1}(1-at)^{-\lambda}(1-bt)^{-\beta}\right) \\
&=\sum_{m=0}^{\infty}\frac{(-\beta)^m\Gamma(\alpha+m)}{m!\Gamma(\alpha)}\cdot\frac{\Gamma(\mu)}{\Gamma(\mu+\alpha+m)}t^{\mu+\alpha+m-1}F_1(\mu,\lambda,\beta,\mu+\alpha+m;at,bt) \\
&=t^{\mu+\alpha-1}\cdot\frac{\Gamma(\mu)}{\Gamma(\mu+\alpha)}\sum_{m=0}^{\infty}\frac{\Gamma(\alpha+m)\Gamma(\mu+\alpha)}{\Gamma(\alpha)\Gamma(\mu+\alpha+m)}\cdot\frac{(-\beta t)^{m}}{m!}F_1(\mu,\lambda,\beta,\mu+\alpha+m;at,bt).
\end{align*}
Similarly to Example \ref{example2}, if the factor $F_1(\mu,\lambda,\beta,\mu+\alpha+m;at,bt)$ were removed from this expression, it would be precisely the series for the confluent hypergeometric function $\prescript{}{1}F_1(\alpha;\mu+\alpha;-\beta t)$. Thus, the tempered fractional integral of this $f(t)$ can be seen as a twisted convolution of the confluent hypergeometric function with the first Appell function.
\end{example}

\begin{example}
\label{example4}
Let $\boldsymbol{f(t)=(t-a)^{\nu-1}E_{\mu,\nu}^{\gamma}\left(\omega(t-a)^{\mu}\right)}$, where $E_{\mu,\nu}^{\gamma}$ is the 3-parameter Mittag-Leffler function \cite{prabhakar} with $\mathrm{Re}(\nu)>0$ and $\mathrm{Re}(\mu)>0$. In this case, using Theorem \ref{Thm:GPFseries} and also the series formula for the Mittag-Leffler function, we have:
\begin{align*}
\prescript{T}{a}I^{(\alpha,\beta)}_tf(t)&=\sum_{m=0}^{\infty}\frac{(-\beta)^m\Gamma(\alpha+m)}{m!\Gamma(\alpha)}\prescript{RL}{a}I^{\alpha+m}_t\left((t-a)^{\nu-1}E_{\mu,\nu}^{\gamma}\left(\omega(t-a)^{\mu}\right)\right) \\
&=\sum_{m=0}^{\infty}\frac{(-\beta)^m\Gamma(\alpha+m)}{m!\Gamma(\alpha)}\prescript{RL}{a}I^{\alpha+m}_t\left(\sum_{k=0}^{\infty}\frac{\Gamma(\gamma+k)\omega^k}{k!\Gamma(\gamma)\Gamma(\mu k+\nu)}(t-a)^{\mu k+\nu-1}\right) \\
&=\sum_{m=0}^{\infty}\frac{(-\beta)^m\Gamma(\alpha+m)}{m!\Gamma(\alpha)}\sum_{k=0}^{\infty}\frac{\Gamma(\gamma+k)\omega^k}{k!\Gamma(\gamma)}\prescript{RL}{a}I^{\alpha+m}_t\left(\frac{(t-a)^{\mu k+\nu-1}}{\Gamma(\mu k+\nu)}\right) \\
&=\sum_{m=0}^{\infty}\frac{(-\beta)^m\Gamma(\alpha+m)}{m!\Gamma(\alpha)}\sum_{k=0}^{\infty}\frac{\Gamma(\gamma+k)\omega^k}{k!\Gamma(\gamma)}\cdot\frac{(t-a)^{\mu k+\nu+\alpha+m-1}}{\Gamma(\mu k+\nu+\alpha+m)}.
\end{align*}
By summing this double series as it is, we obtain
\begin{equation*}
\prescript{T}{a}I^{(\alpha,\beta)}_t\left((t-a)^{\nu-1}E_{\mu,\nu}^{\gamma}\left(\omega(t-a)^{\mu}\right)\right)=\sum_{m=0}^{\infty}\frac{(-\beta)^m\Gamma(\alpha+m)}{m!\Gamma(\alpha)}(t-a)^{\nu+\alpha+m-1}E_{\mu,\nu+\alpha+m}^{\gamma}\left(\omega(t-a)^{\mu}\right).
\end{equation*}
Alternatively, we can rearrange the double series by swapping the order of the $m$-summation and the $k$-summation, which yields:
\begin{align*}
\prescript{T}{a}I^{(\alpha,\beta)}_tf(t)&=\sum_{k=0}^{\infty}\frac{\Gamma(\gamma+k)\omega^k}{k!\Gamma(\gamma)}\sum_{m=0}^{\infty}\frac{(-\beta)^m\Gamma(\alpha+m)}{m!\Gamma(\alpha)}\cdot\frac{(t-a)^{\mu k+\nu+\alpha+m-1}}{\Gamma(\mu k+\nu+\alpha+m)} \\
&=\sum_{k=0}^{\infty}\frac{\Gamma(\gamma+k)\omega^k}{k!\Gamma(\gamma)}(t-a)^{\mu k+\nu+\alpha-1}\sum_{m=0}^{\infty}\frac{\Gamma(\alpha+m)}{m!\Gamma(\alpha)\Gamma(\mu k+\nu+\alpha+m)}\left(-\beta(t-a)\right)^{m} \\
&=\sum_{k=0}^{\infty}\frac{\Gamma(\gamma+k)\omega^k}{k!\Gamma(\gamma)}(t-a)^{\mu k+\nu+\alpha-1}\frac{1}{\Gamma(\mu k+\nu+\alpha)}\prescript{}{1}F_1\left(\alpha;\mu k+\nu+\alpha;-\beta(t-a)\right). \\
\end{align*}
Thus we have two equivalent series expressions for the same function, which yields the following identity between Mittag-Leffler functions and hypergeometric functions:
\begin{multline}
\sum_{m=0}^{\infty}\frac{(-\beta)^m\Gamma(\alpha+m)}{m!\Gamma(\alpha)}(t-a)^{\nu+\alpha+m-1}E_{\mu,\nu+\alpha+m}^{\gamma}\left(\omega(t-a)^{\mu}\right) \\
=\sum_{k=0}^{\infty}\frac{\Gamma(\gamma+k)\omega^k}{k!\Gamma(\gamma)\Gamma(\mu k+\nu+\alpha)}(t-a)^{\mu k+\nu+\alpha-1}\prescript{}{1}F_1\left(\alpha;\mu k+\nu+\alpha;-\beta(t-a)\right),
\end{multline}
valid for $\mu$, $\nu$, $\alpha$, $\beta$ with real parts greater than zero.
\end{example}

\subsection{Mellin transforms}

In order to learn how the Mellin transform interacts with the operators of tempered fractional calculus, we compute the Mellin transform of the tempered fractional differintegral of a general function.

\begin{defn}
The \emph{Mellin transform} of a function $f(t)$ is the function $\widehat{f}(s)$ defined by
\[\widehat{f}(s)=\int_0^{\infty}t^{s-1}f(t)\,\mathrm{d}t.\]
\end{defn}

\begin{theorem}
\label{Thm:Mellin1}
Let $f\in L^1[0,\infty)$ be a function and $\alpha,\beta\in\mathbb{C}$ be two parameters with $\mathrm{Re}(\alpha)>0$. The Mellin transform of the tempered fractional integral of $f$ is given by the following integral transform:
\begin{equation}
\label{Mellin1}
\widehat{\prescript{T}{0}I^{(\alpha,\beta)}f}(s)=\frac{1}{\beta^{\alpha+s-1}\Gamma(\alpha)}\int_0^{\infty}\Gamma_{1-s}\left(\alpha,\beta u\right)f(u)\,\mathrm{d}u,
\end{equation}
where the Kobayashi gamma function $\Gamma_m(u,v)$ is defined in \cite{kobayashi} by
\begin{equation}
\label{Kobayashi}
\Gamma_m(u,v)\coloneqq\int_0^{\infty}\frac{t^{u-1}e^{-t}}{(t+v)^m}\,\mathrm{d}t.
\end{equation}
\end{theorem}

\begin{proof}
By the definitions of the Mellin transform and of the generalised proportional fractional integral,
\begin{align*}
\widehat{\prescript{T}{0}I^{(\alpha,\beta)}f}(s)&=\int_0^{\infty}t^{s-1}\prescript{T}{0}I^{(\alpha,\beta)}_tf(t)\,\mathrm{d}t \\
&=\int_0^{\infty}t^{s-1}\frac{1}{\Gamma(\alpha)}\int_0^t(t-u)^{\alpha-1}e^{-\beta(t-u)}f(u)\,\mathrm{d}u\,\mathrm{d}t \\
&=\frac{1}{\Gamma(\alpha)}\int_0^{\infty}\int_0^tt^{s-1}(t-u)^{\alpha-1}e^{-\beta(t-u)}f(u)\,\mathrm{d}u\,\mathrm{d}t.
\end{align*}
By Fubini's theorem, we swap the order of integration, noting that $0\leq u\leq t\leq\infty$, to get:
\begin{align}
\nonumber \widehat{\prescript{T}{0}I^{(\alpha,\beta)}f}(s)&=\frac{1}{\Gamma(\alpha)}\int_0^{\infty}\int_u^{\infty}t^{s-1}(t-u)^{\alpha-1}e^{-\beta(t-u)}f(u)\,\mathrm{d}t\,\mathrm{d}u \\
\label{Mellin:forlater} &=\frac{1}{\Gamma(\alpha)}\int_0^{\infty}f(u)\int_u^{\infty}t^{s-1}(t-u)^{\alpha-1}e^{-\beta(t-u)}\,\mathrm{d}t\,\mathrm{d}u.
\end{align}
Now we change variables in the inner integral by setting $x=t-u$:
\begin{align*}
\widehat{\prescript{T}{0}I^{(\alpha,\beta)}f}(s)&=\frac{1}{\Gamma(\alpha)}\int_0^{\infty}f(u)\int_0^{\infty}(x+u)^{s-1}x^{\alpha-1}e^{-\beta x}\,\mathrm{d}x\,\mathrm{d}u.
\end{align*}
This is already in the form $\int_0^{\infty}f(u)\Phi(u)\,\mathrm{d}u$, and we just need to show that $\Phi$ can be expressed in terms of the Kobayashi gamma function as required. We make the substitution $x=\frac{t}{\beta}$ and proceed as follows:
\begin{align*}
\Phi(u)&\coloneqq\frac{1}{\Gamma(\alpha)}\int_0^{\infty}(x+u)^{s-1}x^{\alpha-1}e^{-\beta x}\,\mathrm{d}x \\
&=\frac{1}{\Gamma(\alpha)}\int_0^{\infty}\left(\frac{t}{\beta}+u\right)^{s-1}\left(\frac{1}{\beta}\right)^{\alpha-1}t^{\alpha-1}e^{-t}\left(\frac{1}{\beta}\right)\,\mathrm{d}t \\
&=\frac{1}{\beta^{\alpha}\Gamma(\alpha)}\int_0^{\infty}\left(\frac{1}{\beta}\right)^{s-1}\left(t+\beta u\right)^{s-1}t^{\alpha-1}e^{-t}\,\mathrm{d}t \\
&=\frac{1}{\beta^{\alpha+s-1}\Gamma(\alpha)}\int_0^{\infty}\left(t+\beta u\right)^{s-1}t^{\alpha-1}e^{-t}\,\mathrm{d}t \\
&=\frac{1}{\beta^{\alpha+s-1}\Gamma(\alpha)}\Gamma_{1-s}\left(\alpha,\beta u\right).
\end{align*}
Then we have
\[\widehat{\prescript{T}{0}I^{(\alpha,\beta)}f}(s)=\int_0^{\infty}f(u)\Phi(u)\,\mathrm{d}u=\frac{1}{\beta^{\alpha+s-1}\Gamma(\alpha)}\int_0^{\infty}f(u)\Gamma_{1-s}\left(\alpha,\beta u\right)\,\mathrm{d}u,\]
as required.
\end{proof}

\begin{theorem}
\label{Thm:Mellin2}
With all notation as in Theorem \ref{Thm:Mellin1}, the Mellin transform of the generalised proportional fractional integral of $f$ can be rewritten as:
\begin{equation}
\label{Mellin2}
\widehat{\prescript{T}{0}I^{(\alpha,\beta)}f}(s)=\beta^{-\alpha-s+1}\int_0^{\infty}f(u)e^{\beta u}\sum_{n=0}^{\infty}\frac{\Gamma\left(\alpha+s-n-1,\beta t\right)}{\Gamma(\alpha-n)}\cdot\frac{(\beta u)^n}{n!}\,\mathrm{d}u,
\end{equation}
\end{theorem}

\begin{proof}
Instead of starting from the expression \eqref{Mellin1} for the Mellin transform of the tempered fractional integral, we go back to the previous expression \eqref{Mellin:forlater}. The reason for this is that, when we have variables $t$ and $u$ with $t\geq u$, it is possible to use the binomial series for $(t-u)^{\alpha-1}$. By contrast, with $x$ and $u$ independently varying from $0$ to $\infty$, it is impossible to use the binomial series for $(x+u)^{s-1}$.
\begin{align*}
\widehat{\prescript{T}{0}I^{(\alpha,\beta)}f}(s)&=\frac{1}{\Gamma(\alpha)}\int_0^{\infty}f(u)\int_u^{\infty}t^{s-1}(t-u)^{\alpha-1}e^{-\beta(t-u)}\,\mathrm{d}t\,\mathrm{d}u \\
&=\frac{1}{\Gamma(\alpha)}\int_0^{\infty}f(u)\int_u^{\infty}t^{s-1}\left[\sum_{n=0}^{\infty}\frac{\Gamma(\alpha)}{\Gamma(\alpha-n)n!}u^nt^{\alpha-1-n}\right]e^{-\beta t}e^{\beta u}\,\mathrm{d}t\,\mathrm{d}u \\
&=\frac{1}{\Gamma(\alpha)}\int_0^{\infty}f(u)e^{\beta u}\int_u^{\infty}\sum_{n=0}^{\infty}\frac{\Gamma(\alpha)}{\Gamma(\alpha-n)n!}u^nt^{\alpha+s-n-2}e^{-\beta t}\,\mathrm{d}t\,\mathrm{d}u.
\end{align*}
The series here is locally uniformly convergent for $t\geq u$, so we can swap the summation and integration to get:
\begin{align*}
\widehat{\prescript{T}{0}I^{(\alpha,\beta)}f}(s)&=\frac{1}{\Gamma(\alpha)}\int_0^{\infty}f(u)e^{\beta u}\sum_{n=0}^{\infty}\frac{\Gamma(\alpha)}{\Gamma(\alpha-n)n!}u^n\int_u^{\infty}t^{\alpha+s-n-2}e^{-\beta t}\,\mathrm{d}t\,\mathrm{d}u \\
&=\int_0^{\infty}f(u)e^{\beta u}\sum_{n=0}^{\infty}\frac{1}{\Gamma(\alpha-n)n!}u^n\int_{\beta u}^{\infty}v^{\alpha+s-n-2}e^{-v}\left(\frac{1}{\beta}\right)^{\alpha+s-n-1}\,\mathrm{d}v\,\mathrm{d}u \\
&=\int_0^{\infty}f(u)e^{\beta u}\sum_{n=0}^{\infty}\frac{1}{\Gamma(\alpha-n)n!}u^n\left(\frac{1}{\beta}\right)^{\alpha+s-n-1}\Gamma\left(\alpha+s-n-1,\beta u\right)\,\mathrm{d}u \\
&=\frac{1}{\beta^{\alpha+s-1}}\int_0^{\infty}f(u)e^{\beta u}\sum_{n=0}^{\infty}\frac{1}{\Gamma(\alpha-n)n!}(\beta u)^n\Gamma\left(\alpha+s-n-1,\beta u\right)\,\mathrm{d}u.
\end{align*}
And this is precisely \eqref{Mellin2}.
\end{proof}

\section{Taylor's theorem for tempered fractional derivatives} \label{sec:Taylor}

In this section, we follow the methodology used in \cite{fernandez-baleanu,odibat-shawagfeh} to prove a version of Taylor's theorem which is valid for tempered fractional derivatives.

We begin with the following lemma concerning compositions of repeated tempered integrals and derivatives.

\begin{lem}
\label{Lem:mcomposn}
Let $[a,b]$ be a real interval and $\alpha,\beta\in\mathbb{C}$ with $\mathrm{Re}(\alpha)>0$. If $f$ is a smooth function on $[a,b]$, then for any $r\in\mathbb{N}$,
\begin{multline}
\label{mcomposn:result}
\left(\prescript{T}{a}I^{(\alpha,\beta)}_t\right)^r\left(\prescript{T}{a}D^{(\alpha,\beta)}_t\right)^rf(t)-\left(\prescript{T}{a}I^{(\alpha,\beta)}_t\right)^{r+1}\left(\prescript{T}{a}D^{(\alpha,\beta)}_t\right)^{r+1}f(t) \\ =-e^{-\beta(t-a)}\sum_{k=1}^n\frac{(t-a)^{r\alpha+\alpha-k}}{\Gamma(r\alpha+\alpha-k+1)}\prescript{T}{a}I_t^{(k-\alpha,\beta)}\left(\prescript{T}{a}D_t^{(\alpha,\beta)}\right)^rf(a^+),
\end{multline}
where $n$ is defined by \eqref{n:defn}.
\end{lem}

\begin{proof}
For the duration of this proof, we shall drop some of the labels from our operators, writing simply $I^{(\alpha,\beta)}$ and $D^{(\alpha,\beta)}$ instead of $\prescript{T}{a}I^{(\alpha,\beta)}_t$ and $\prescript{T}{a}D^{(\alpha,\beta)}_t$.

We recall from Proposition \ref{Prop:semigroup} that
\[\left(1-I^{(\alpha,\beta)}\circ D^{(\alpha,\beta)}\right)f(t)=-e^{-\beta(t-a)}\sum_{k=1}^n\frac{(t-a)^{\alpha-k}}{\Gamma(\alpha-k+1)}I^{(k-\alpha,\beta)}f(a^+),\]
where $n=\lfloor\mathrm{Re}(\alpha)\rfloor+1$ is the smallest natural number greater than $\mathrm{Re}(\alpha)$. With some manipulation of operators, this identity can be substituted into the left-hand side of \eqref{mcomposn:result}:
\begin{align*}
&\left(I^{(\alpha,\beta)}\right)^r\left(D^{(\alpha,\beta)}\right)^rf(t)-\left(I^{(\alpha,\beta)}\right)^{r+1}\left(D^{(\alpha,\beta)}\right)^{r+1}f(t) \\
&\hspace{2cm}=\left(I^{(\alpha,\beta)}\right)^r\left(1-I^{(\alpha,\beta)}\circ D^{(\alpha,\beta)}\right)\left(D^{(\alpha,\beta)}\right)^rf(t) \\
&\hspace{2cm}=\left(I^{(\alpha,\beta)}\right)^r\left[-e^{-\beta(t-a)}\sum_{k=1}^n\frac{(t-a)^{\alpha-k}}{\Gamma(\alpha-k+1)}I^{(k-\alpha,\beta)}\left[\left(D^{(\alpha,\beta)}\right)^rf\right](a^+)\right].
\end{align*}
It is important to note that the quantity
\[A_k\coloneqq I^{(k-\alpha,\beta)}\left[\left(D^{(\alpha,\beta)}\right)^rf\right](a^+)\]
is a constant independent of $t$, because evaluating at the point $t=a$ eliminates the $t$-dependence, and therefore it can be left out of all differintegration manipulations. Continuing:
\begin{align*}
&\left(I^{(\alpha,\beta)}\right)^r\left(D^{(\alpha,\beta)}\right)^rf(t)-\left(I^{(\alpha,\beta)}\right)^{r+1}\left(D^{(\alpha,\beta)}\right)^{r+1}f(t) \\
&\hspace{2cm}=\left(I^{(\alpha,\beta)}\right)^r\left[-e^{-\beta(t-a)}\sum_{k=1}^n\frac{(t-a)^{\alpha-k}}{\Gamma(\alpha-k+1)}A_k\right] \\
&\hspace{2cm}=\sum_{k=1}^n\frac{A_k}{\Gamma(\alpha-k+1)}\left(I^{(\alpha,\beta)}\right)^r\left[-e^{-\beta(t-a)}(t-a)^{\alpha-k}\right] \\
&\hspace{2cm}=\sum_{k=1}^n\frac{-A_ke^{\beta a}}{\Gamma(\alpha-k+1)}I^{(r\alpha,\beta)}\left[e^{-\beta t}(t-a)^{\alpha-k}\right]
\end{align*}
It is known \cite[Proposition 3.7]{jarad-abdeljawad-alzabut} that
\[I^{(\alpha,\beta)}\left(e^{-\beta t}(t-a)^{\beta-1}\right)=\frac{\Gamma(\beta)}{\Gamma(\beta+\alpha)}e^{-\beta t}(t-a)^{\alpha+\beta-1},\quad\quad\mathrm{Re}(\alpha)>0,\mathrm{Re}(\beta)>0.\]
Applying this identity with $\alpha$ and $\beta$ replaced respectively by $r\alpha$ and $\alpha-k+1$, we find:
\begin{align*}
&\left(I^{(\alpha,\beta)}\right)^r\left(D^{(\alpha,\beta)}\right)^rf(t)-\left(I^{(\alpha,\beta)}\right)^{r+1}\left(D^{(\alpha,\beta)}\right)^{r+1}f(t) \\
&\hspace{2cm}=\sum_{k=1}^n\frac{-A_ke^{\beta a}}{\Gamma(\alpha-k+1)}\cdot\frac{\Gamma(\alpha-k+1)}{\Gamma(\alpha-k+1+r\alpha)}e^{-\beta t}(t-a)^{r\alpha+\alpha-k+1-1} \\
&\hspace{2cm}=\sum_{k=1}^n\frac{-A_k}{\Gamma(r\alpha+\alpha-k+1)}e^{-\beta(t-a)}(t-a)^{r\alpha+\alpha-k} \\
&\hspace{2cm}=-e^{-\beta(t-a)}\sum_{k=1}^n\frac{A_k(t-a)^{r\alpha+\alpha-k}}{\Gamma(r\alpha+\alpha-k+1)},
\end{align*}
which gives the desired result. Note that the series $\sum_{k=1}^n$ appearing here looks like a finite truncation of the power series for a Mittag-Leffler function with 2 parameters.
\end{proof}

\begin{coroll}
\label{Cor:mcomposn:a<1}
Let $[a,b]$ be a real interval, $\beta\in\mathbb{C}$, and $\alpha\in(0,1)$. If $f$ is a smooth function on $[a,b]$, then for any $r\in\mathbb{N}$,
\begin{multline}
\label{mcomposn:a<1:result}
\left(\prescript{T}{a}I^{(\alpha,\beta)}_t\right)^r\left(\prescript{T}{a}D^{(\alpha,\beta)}_t\right)^rf(t)-\left(\prescript{T}{a}I^{(\alpha,\beta)}_t\right)^{r+1}\left(\prescript{T}{a}D^{(\alpha,\beta)}_t\right)^{r+1}f(t) \\ =-e^{-\beta(t-a)}\frac{(t-a)^{r\alpha+\alpha-1}}{\Gamma(r\alpha+\alpha)}\prescript{T}{a}I_t^{(1-\alpha,\beta)}\left(\prescript{T}{a}D_t^{(\alpha,\beta)}\right)^rf(a^+).
\end{multline}
\end{coroll}

\begin{proof}
If $\alpha\in(0,1)$, then $n=1$ as defined by \eqref{n:defn}, so the series $\sum_{k=1}^n$ in the result of Lemma \ref{Lem:mcomposn} becomes a single term with just $k=1$.
\end{proof}

\begin{theorem} \label{Thm:Taylor}
Let $[a,b]$ be a real interval and $\alpha,\beta\in\mathbb{C}$ with $\mathrm{Re}(\alpha)>0$. If $f$ is a smooth function on $[a,b]$, then for any $m\in\mathbb{N}$ and $t\in[a,b]$, there exists $c\in(a,t)$ such that
\begin{multline}
\label{Taylor:eqn}
f(t)=\sum_{r=0}^m\left[\sum_{k=1}^n\frac{-(t-a)^{r\alpha+\alpha-k}e^{-\beta(t-a)}}{\Gamma(r\alpha+\alpha-k+1)}\prescript{T}{a}I_t^{(k-\alpha,\beta)}\right]\left(\prescript{T}{a}D_t^{(\alpha,\beta)}\right)^rf(a^+) \\ +\frac{\gamma\left((m+1)\alpha,\beta(t-a)\right)}{\beta^{(m+1)\alpha}\Gamma((m+1)\alpha)}\left(\prescript{T}{a}D_t^{(\alpha,\beta)}\right)^{m+1}f(c),
\end{multline}
where $n$ is defined by \eqref{n:defn}.
\end{theorem}

\begin{proof}
We sum the result \eqref{mcomposn:result} of Lemma \ref{Lem:mcomposn} over all values of $r$ from zero up to $m$. The left-hand side then becomes a telescoping series, and we get:
\begin{multline} \label{Taylor:telescope}
f(t)-\left(\prescript{T}{a}I^{(\alpha,\beta)}_t\right)^{m+1}\left(\prescript{T}{a}D^{(\alpha,\beta)}_t\right)^{m+1}f(t) \\ =-e^{-\beta(t-a)}\sum_{r=0}^m\sum_{k=1}^n\frac{(t-a)^{r\alpha+\alpha-k}}{\Gamma(r\alpha+\alpha-k+1)}\prescript{T}{a}I_t^{(k-\alpha,\beta)}\left(\prescript{T}{a}D_t^{(\alpha,\beta)}\right)^rf(a^+).
\end{multline}
It remains to consider the term $\left(I^{(\alpha,\beta)}\right)^{m+1}\left(D^{(\alpha,\beta)}\right)^{m+1}f(t)$, where again we drop several of the labels of the $I$ and $D$ operators for convenience of notation. Using the result of Proposition \ref{Prop:semigroup}, we can rewrite this function as follows:
\begin{align*}
\left(I^{(\alpha,\beta)}\right)^{m+1}\left(D^{(\alpha,\beta)}\right)^{m+1}f(t)&=I^{((m+1)\alpha,\beta)}\left(D^{(\alpha,\beta)}\right)^{m+1}f(t) \\ &=\frac{1}{\Gamma((m+1)\alpha)}\int_a^t(t-u)^{(m+1)\alpha-1}e^{-\beta(t-u)}\left(D^{(\alpha,\beta)}\right)^{m+1}f(u)\,\mathrm{d}u.
\end{align*}
Since $f$ is smooth, $\left(D^{(\alpha,\beta)}\right)^{m+1}f(u)$ is certainly a continuous function of $u$. And the other part of the integrand, the function $(t-u)^{(m+1)\alpha-1}e^{-\beta(t-u)}$, is integrable and positive for $u\in(a,t)$. So by the mean value theorem for integrals, there exists $c\in(a,t)$ such that
\begin{equation*}
\left(I^{(\alpha,\beta)}\right)^{m+1}\left(D^{(\alpha,\beta)}\right)^{m+1}f(t)=\frac{1}{\Gamma((m+1)\alpha)}\left(D^{(\alpha,\beta)}\right)^{m+1}f(c)\int_a^t(t-u)^{(m+1)\alpha-1}e^{-\beta(t-u)}\,\mathrm{d}u.
\end{equation*}
Setting $v=t-u$ to get an integral starting from zero, and then setting $w=\beta v$ to get a simple exponential function in the integrand:
\begin{align*}
\left(I^{(\alpha,\beta)}\right)^{m+1}\left(D^{(\alpha,\beta)}\right)^{m+1}f(t)&=\frac{\left(D^{(\alpha,\beta)}\right)^{m+1}f(c)}{\Gamma((m+1)\alpha)}\int_{t-a}^0v^{(m+1)\alpha-1}e^{-\beta v}(-1)\,\mathrm{d}v \\
&=\frac{\left(D^{(\alpha,\beta)}\right)^{m+1}f(c)}{\Gamma((m+1)\alpha)}\int^{t-a}_0v^{(m+1)\alpha-1}e^{-\beta v}\,\mathrm{d}v \\
&=\frac{\left(D^{(\alpha,\beta)}\right)^{m+1}f(c)}{\Gamma((m+1)\alpha)}\int^{\beta(t-a)}_0\left(\frac{w}{\beta}\right)^{(m+1)\alpha-1}e^{-w}\left(\frac{1}{\beta}\right)\,\mathrm{d}w \\
&=\frac{\left(D^{(\alpha,\beta)}\right)^{m+1}f(c)}{\beta^{(m+1)\alpha}\Gamma((m+1)\alpha)}\int^{\beta(t-a)}_0w^{(m+1)\alpha-1}e^{-w}\,\mathrm{d}w \\
&=\frac{\gamma\left((m+1)\alpha,\beta(t-a)\right)}{\beta^{(m+1)\alpha}\Gamma((m+1)\alpha)}\left(D^{(\alpha,\beta)}\right)^{m+1}f(c).
\end{align*}
Substituting this into the identity \eqref{Taylor:telescope}, we find the desired result.
\end{proof}

\begin{coroll}
\label{Cor:Taylor:a<1}
Let $[a,b]$ be a real interval, $\beta\in\mathbb{C}$, and $\alpha\in(0,1)$. If $f$ is a smooth function on $[a,b]$, then for any $r\in\mathbb{N}$ and $t\in[a,b]$, there exists $c\in(a,t)$ such that
\begin{multline}
\label{Taylor:a<1:eqn}
f(t)=\sum_{r=0}^m\frac{-(t-a)^{r\alpha+\alpha-1}e^{-\beta(t-a)}}{\Gamma(r\alpha+\alpha)}\prescript{T}{a}I_t^{(1-\alpha,\beta)}\left(\prescript{T}{a}D_t^{(\alpha,\beta)}\right)^rf(a^+) \\ +\frac{\gamma\left((m+1)\alpha,\beta(t-a)\right)}{\beta^{(m+1)\alpha}\Gamma((m+1)\alpha)}\left(\prescript{T}{a}D_t^{(\alpha,\beta)}\right)^{m+1}f(c),
\end{multline}
\end{coroll}

\begin{proof}
If $\alpha\in(0,1)$, then $n=1$ as defined by \eqref{n:defn}, so the series $\sum_{k=1}^n$ in the result of Theorem \ref{Thm:Taylor} becomes a single term with just $k=1$.
\end{proof}

Theorem \ref{Thm:Taylor} and its special case Corollary \ref{Cor:Taylor:a<1} demonstrate a version of Taylor series which works for tempered fractional derivatives. Note that this is Taylor's theorem for a finite series with a remainder. In order to obtain an infinite Taylor series of tempered fractional derivatives, we would need to impose a convergence condition on the series \eqref{Taylor:eqn}. This equates to finding a condition for ``fractional analyticity'', the existence of a convergent fractional Taylor series, in the context of tempered fractional calculus. Such a task is beyond the scope of the current paper, but it could be achieved in future works in the same direction.

\section{Some inequalities for tempered fractional integrals} \label{sec:intineq}

In this section, we detail some integral inequalities concerning tempered fractional integrals, which are analogues of corresponding results in other models of fractional calculus \cite{belarbi-dahmani} and which may be useful in understanding the structure of tempered fractional calculus.

\begin{defn}
\label{Def:sync}
Two real functions $f,g$ defined on an interval $[a,b]\subset\mathbb{R}$ are said to be \emph{synchronous} if the following equality holds for all $u,v\in[a,b]$:
\begin{equation}
\label{sync:eqn}
\left(f(u)-f(v)\right)\left(g(u)-g(v)\right)\geq0.
\end{equation}
Intuitively, this means that the regions of increasing or decreasing are synchronised across both functions, i.e. $f$ and $g$ go up and down together.
\end{defn}

\begin{theorem}
\label{Thm:intineq1}
If $f,g\in L^1[0,\infty)$ are two synchronous functions and $\alpha,\beta\in\mathbb{C}$ are two parameters with $\mathrm{Re}(\alpha)>0$, then the following integral inequality is valid for all $t\geq0$:
\begin{equation}
\label{intineq1}
\prescript{T}{0}I_t^{(\alpha,\beta)}[f(t)g(t)]\geq\left[\frac{\beta^{\alpha}\Gamma(\alpha)}{\gamma(\alpha,\beta t)}\right]\prescript{T}{0}I_t^{(\alpha,\beta)}[f(t)]\prescript{T}{0}I_t^{(\alpha,\beta)}[g(t)].
\end{equation}
\end{theorem}

\begin{proof}
Starting from the definition \eqref{sync:eqn} of synchronous functions, we can find the following inequality on $f$ and $g$:
\[f(u)g(u)+f(v)g(v)\geq f(u)g(v)+g(u)f(v),\quad\quad u,v\in[0,\infty).\]
We multiply this inequality by the positive factor $\frac{1}{\Gamma(\alpha)}e^{-\beta(t-u)}(t-u)^{\alpha-1}$ on both left and right hand sides, and then integrate both sides with respect to $u$ over the interval $(0,t)$:
\begin{align}
\nonumber &{\color{white}\Rightarrow}\frac{1}{\Gamma(\alpha)}\int_0^te^{-\beta(t-u)}(t-u)^{\alpha-1}\Big[f(u)g(u)+f(v)g(v)\Big]\,\mathrm{d}u \\ 
\nonumber &\hspace{3cm}\geq\frac{1}{\Gamma(\alpha)}\int_0^te^{-\beta(t-u)}(t-u)^{\alpha-1}\Big[f(u)g(v)+g(u)f(v)\Big]\,\mathrm{d}u \\
\nonumber &\Rightarrow\frac{1}{\Gamma(\alpha)}\int_0^te^{-\beta(t-u)}(t-u)^{\alpha-1}f(u)g(u)\,\mathrm{d}u+\frac{f(v)g(v)}{\Gamma(\alpha)}\int_0^te^{-\beta(t-u)}(t-u)^{\alpha-1}\,\mathrm{d}u \\ 
\nonumber &\hspace{3cm}\geq\frac{g(v)}{\Gamma(\alpha)}\int_0^te^{-\beta(t-u)}(t-u)^{\alpha-1}f(u)\Big]\,\mathrm{d}u+\frac{f(v)}{\Gamma(\alpha)}\int_0^te^{-\beta(t-u)}(t-u)^{\alpha-1}g(u)\,\mathrm{d}u \\
\label{intineq1.5} &\Rightarrow\prescript{T}{0}I_t^{(\alpha,\beta)}\Big[f(t)g(t)\Big]+f(v)g(v)\prescript{T}{0}I_t^{(\alpha,\beta)}\Big[1\Big]\geq g(v)\prescript{T}{0}I_t^{(\alpha,\beta)}\Big[f(t)\Big]+f(v)\prescript{T}{0}I_t^{(\alpha,\beta)}\Big[g(t)\Big],
\end{align}
valid for all $t,v\in[0,\infty)$. Now we multiply \eqref{intineq1.5} by the positive factor $\frac{1}{\Gamma(\alpha)}e^{-\beta(t-v)}(t-v)^{\alpha-1}$ on both left and right hand sides, and then integrate both sides with respect to $v$ over the interval $(0,t)$:
\begin{align*}
&{\color{white}\Rightarrow}\frac{1}{\Gamma(\alpha)}\int_0^te^{-\beta(t-v)}(t-v)^{\alpha-1}\left[\prescript{T}{0}I_t^{(\alpha,\beta)}[f(t)g(t)]+f(v)g(v)\prescript{T}{0}I_t^{(\alpha,\beta)}[1]\right]\,\mathrm{d}v \\ &\hspace{2cm}\geq\frac{1}{\Gamma(\alpha)}\int_0^te^{-\beta(t-v)}(t-v)^{\alpha-1}\left[g(v)\prescript{T}{0}I_t^{(\alpha,\beta)}[f(t)]+f(v)\prescript{T}{0}I_t^{(\alpha,\beta)}[g(t)]\right]\,\mathrm{d}v \\
&\Rightarrow\frac{\prescript{T}{0}I_t^{(\alpha,\beta)}[f(t)g(t)]}{\Gamma(\alpha)}\int_0^te^{-\beta(t-v)}(t-v)^{\alpha-1}\,\mathrm{d}v+\frac{\prescript{T}{0}I_t^{(\alpha,\beta)}[1]}{\Gamma(\alpha)}\int_0^te^{-\beta(t-v)}(t-v)^{\alpha-1}f(v)g(v)\,\mathrm{d}v \\ &\hspace{2cm}\geq\frac{\prescript{T}{0}I_t^{(\alpha,\beta)}[f(t)]}{\Gamma(\alpha)}\int_0^te^{-\beta(t-v)}(t-v)^{\alpha-1}g(v)\,\mathrm{d}v+\frac{\prescript{T}{0}I_t^{(\alpha,\beta)}[g(t)]}{\Gamma(\alpha)}\int_0^te^{-\beta(t-v)}(t-v)^{\alpha-1}f(v)\,\mathrm{d}v \\
&\Rightarrow\prescript{T}{0}I_t^{(\alpha,\beta)}[f(t)g(t)]\prescript{T}{0}I_t^{(\alpha,\beta)}\Big[1\Big]+\prescript{T}{0}I_t^{(\alpha,\beta)}[1]\prescript{T}{0}I_t^{(\alpha,\beta)}\Big[f(t)g(t)\Big] \\ &\hspace{2cm}\geq\prescript{T}{0}I_t^{(\alpha,\beta)}[f(t)]\prescript{T}{0}I_t^{(\alpha,\beta)}\Big[g(t)\Big]+\prescript{T}{0}I_t^{(\alpha,\beta)}[g(t)]\prescript{T}{0}I_t^{(\alpha,\beta)}\Big[f(t)\Big] \\
&\Rightarrow2\prescript{T}{0}I_t^{(\alpha,\beta)}[f(t)g(t)]\prescript{T}{0}I_t^{(\alpha,\beta)}[1]\geq2\prescript{T}{0}I_t^{(\alpha,\beta)}[f(t)]\prescript{T}{0}I_t^{(\alpha,\beta)}[g(t)].
\end{align*}
Now, to achieve the desired inequality, we just need to show that
\begin{equation}
\label{JAA1}
\prescript{T}{0}I_t^{(\alpha,\beta)}[1]=\frac{\gamma\left(\alpha,\beta t\right)}{\beta^{\alpha}\Gamma(\alpha)},
\end{equation}
and this follows directly from the definition \eqref{Tint:def}:
\begin{align*}
\prescript{T}{0}I^{(\alpha,\beta)}_t[1]&=\frac{1}{\Gamma(\alpha)}\int_0^t(t-u)^{\alpha-1}e^{-\beta(t-u)}\,\mathrm{d}u \\
&=\frac{1}{\Gamma(\alpha)}\int_0^tu^{\alpha-1}e^{-\beta u}\,\mathrm{d}u \\
&=\frac{1}{\Gamma(\alpha)}\int_0^{\beta t}v^{\alpha-1}e^{-v}\left(\frac{1}{\beta}\right)^{\alpha}\,\mathrm{d}v \\
&=\frac{1}{\beta^{\alpha}\Gamma(\alpha)}\gamma\left(\alpha,\beta t\right).
\end{align*}
\end{proof}

\begin{theorem}
\label{Thm:intineq2}
If $f,g\in L^1[0,\infty)$ are two synchronous functions and $\alpha_1,\alpha_2,\beta\in\mathbb{C}$ are three parameters with $\mathrm{Re}(\alpha_1),\mathrm{Re}(\alpha_2)>0$, then the following integral inequality is valid for all $t>0$:
\begin{multline}
\label{intineq2}
\frac{\gamma(\alpha_2,\beta t)}{\beta^{\alpha_2}\Gamma(\alpha_2)}\prescript{T}{0}I_t^{(\alpha_1,\beta)}[f(t)g(t)]+\frac{\gamma(\alpha_1,\beta t)}{\beta^{\alpha_1}\Gamma(\alpha_1)}\prescript{T}{0}I_t^{(\alpha_2,\beta)}[f(t)g(t)] \\ \geq\prescript{T}{0}I_t^{(\alpha_1,\beta)}[f(t)]\prescript{T}{0}I_t^{(\alpha_2,\beta)}[g(t)]+\prescript{T}{0}I_t^{(\alpha_1,\beta)}[g(t)]\prescript{T}{0}I_t^{(\alpha_2,\beta)}[f(t)].
\end{multline}
\end{theorem}

\begin{proof}
We recall from the proof of Theorem \ref{intineq1} the following inequality \eqref{intineq1.5}, valid for all $t,v\in[0,\infty)$:
\[\prescript{T}{0}I_t^{(\alpha_1,\beta)}\Big[f(t)g(t)\Big]+f(v)g(v)\prescript{T}{0}I_t^{(\alpha_1,\beta)}\Big[1\Big]\geq g(v)\prescript{T}{0}I_t^{(\alpha_1,\beta)}\Big[f(t)\Big]+f(v)\prescript{T}{0}I_t^{(\alpha_1,\beta)}\Big[g(t)\Big].\]
Now, instead of multiplying this by $\frac{1}{\Gamma(\alpha_1)}e^{-\beta(t-v)}(t-v)^{\alpha_1-1}$ as before, we multiply the left and right hand sides by $\frac{1}{\Gamma(\alpha_2)}e^{-\beta(t-v)}(t-v)^{\alpha_2-1}$ and then integrate both sides with respect to $v$ over the interval $(0,t)$:
\begin{align*}
&{\color{white}\Rightarrow}\frac{1}{\Gamma(\alpha_2)}\int_0^te^{-\beta(t-v)}(t-v)^{\alpha_2-1}\left[\prescript{T}{0}I_t^{(\alpha_1,\beta)}[f(t)g(t)]+f(v)g(v)\prescript{T}{0}I_t^{(\alpha_1,\beta)}[1]\right]\,\mathrm{d}v \\ &\hspace{2cm}\geq\frac{1}{\Gamma(\alpha_2)}\int_0^te^{-\beta(t-v)}(t-v)^{\alpha_2-1}\left[g(v)\prescript{T}{0}I_t^{(\alpha_1,\beta)}[f(t)]+f(v)\prescript{T}{0}I_t^{(\alpha_1,\beta)}[g(t)]\right]\,\mathrm{d}v \\
&\Rightarrow\frac{\prescript{T}{0}I_t^{(\alpha_1,\beta)}[f(t)g(t)]}{\Gamma(\alpha_2)}\int_0^te^{-\beta(t-v)}(t-v)^{\alpha_2-1}\,\mathrm{d}v+\frac{\prescript{T}{0}I_t^{(\alpha_1,\beta)}[1]}{\Gamma(\alpha_2)}\int_0^te^{-\beta(t-v)}(t-v)^{\alpha_2-1}f(v)g(v)\,\mathrm{d}v \\ &\hspace{2cm}\geq\frac{\prescript{T}{0}I_t^{(\alpha_1,\beta)}[f(t)]}{\Gamma(\alpha_2)}\int_0^te^{-\beta(t-v)}(t-v)^{\alpha_2-1}g(v)\,\mathrm{d}v+\frac{\prescript{T}{0}I_t^{(\alpha_1,\beta)}[g(t)]}{\Gamma(\alpha_2)}\int_0^te^{-\beta(t-v)}(t-v)^{\alpha_2-1}f(v)\,\mathrm{d}v \\
&\Rightarrow\prescript{T}{0}I_t^{(\alpha_1,\beta)}[f(t)g(t)]\prescript{T}{0}I_t^{(\alpha_2,\beta)}\Big[1\Big]+\prescript{T}{0}I_t^{(\alpha_1,\beta)}[1]\prescript{T}{0}I_t^{(\alpha_2,\beta)}\Big[f(t)g(t)\Big] \\ &\hspace{4cm}\geq\prescript{T}{0}I_t^{(\alpha_1,\beta)}[f(t)]\prescript{T}{0}I_t^{(\alpha_2,\beta)}\Big[g(t)\Big]+\prescript{T}{0}I_t^{(\alpha_1,\beta)}[g(t)]\prescript{T}{0}I_t^{(\alpha_2,\beta)}\Big[f(t)\Big].
\end{align*}
And this reduces to the required result \eqref{intineq2} when we recall the formula \eqref{JAA1} for the tempered fractional integral of the unit function $1$.
\end{proof}

\begin{remark}
We note that Theorem \ref{Thm:intineq1} is a special case of Theorem \ref{Thm:intineq2}: after setting $\alpha_2=\alpha_1$, the inequality \eqref{intineq2} becomes precisely the inequality \eqref{intineq1}. However, Theorem \ref{Thm:intineq1} is still useful in its own right, as shown for example by its consequence in the following Theorem \ref{Thm:intineq3}.
\end{remark}

\begin{theorem}
\label{Thm:intineq3}
If $f_1,f_2,\dots,f_n\in L^1[0,\infty)$ are $n$ positive increasing functions ($n\in\mathbb{N}$) and $\alpha,\beta\in\mathbb{C}$ are two parameters with $\mathrm{Re}(\alpha)>0$, then the following integral inequality is valid for all $t\geq0$:
\begin{equation}
\label{intineq3}
\prescript{T}{0}I_t^{(\alpha,\beta)}\left(\prod_{i=1}^nf_i(t)\right)\geq\left[\frac{\beta^{\alpha}\Gamma(\alpha)}{\gamma(\alpha,\beta t)}\right]^{n-1}\prod_{i=1}^n\bigg(\prescript{T}{0}I_t^{(\alpha,\beta)}f_i(t)\bigg).
\end{equation}
\end{theorem}

\begin{proof}
We proceed by induction. The case $n=1$ is trivial: indeed, the inequality is actually an equality in this case. The case $n=2$ is precisely the result of Theorem \ref{Thm:intineq1}.

Let us now assume the case $n=k$, namely
\begin{equation}
\label{intineq3:indhyp}
\prescript{T}{0}I_t^{(\alpha,\beta)}\left(\prod_{i=1}^kf_i(t)\right)\geq\left[\frac{\beta^{\alpha}\Gamma(\alpha)}{\gamma(\alpha,\beta t)}\right]^{k-1}\prod_{i=1}^k\bigg(\prescript{T}{0}I_t^{(\alpha,\beta)}f_i(t)\bigg),
\end{equation}
and deduce from this the case $n=k+1$. We use the result of Theorem \ref{Thm:intineq1} with $f=f_{k+1}$ and $g=\prod_{i=1}^kf_i$, namely
\begin{align*}
\prescript{T}{0}I_t^{(\alpha,\beta)}\left[f_{k+1}(t)\prod_{i=1}^kf_i(t)\right]\geq\left[\frac{\beta^{\alpha}\Gamma(\alpha)}{\gamma(\alpha,\beta t)}\right]\prescript{T}{0}I_t^{(\alpha,\beta)}\bigg[f_{k+1}(t)\bigg]\prescript{T}{0}I_t^{(\alpha,\beta)}\left[\prod_{i=1}^kf_i(t)\right].
\end{align*}
Using first this inequality and then the induction hypothesis \eqref{intineq3:indhyp}, we get:
\begin{align*}
\prescript{T}{0}I_t^{(\alpha,\beta)}\left[\prod_{i=1}^{k+1}f_i(t)\right]&\geq\left[\frac{\beta^{\alpha}\Gamma(\alpha)}{\gamma(\alpha,\beta t)}\right]\prescript{T}{0}I_t^{(\alpha,\beta)}\bigg[f_{k+1}(t)\bigg]\prescript{T}{0}I_t^{(\alpha,\beta)}\left[\prod_{i=1}^kf_i(t)\right] \\
&\geq \left[\frac{\beta^{\alpha}\Gamma(\alpha)}{\gamma(\alpha,\beta t)}\right]\prescript{T}{0}I_t^{(\alpha,\beta)}\bigg[f_{k+1}(t)\bigg]\left[\frac{\beta^{\alpha}\Gamma(\alpha)}{\gamma(\alpha,\beta t)}\right]^{k-1}\prod_{i=1}^k\bigg(\prescript{T}{0}I_t^{(\alpha,\beta)}f_i(t)\bigg) \\
&=\left[\frac{\beta^{\alpha}\Gamma(\alpha)}{\gamma(\alpha,\beta t)}\right]^{k}\prod_{i=1}^{k+1}\bigg(\prescript{T}{0}I_t^{(\alpha,\beta)}f_i(t)\bigg),
\end{align*}
which is the result \eqref{intineq3} with $n=k+1$, as required.
\end{proof}

\section{Conclusions} \label{sec:concl}

In this paper, we have performed some detailed analysis of the mathematical underpinnings of tempered fractional calculus. Starting from the definitions and results already established in the literature, we established many new properties of the tempered fractional integrals and derivatives.

The connections with the Riemann--Liouville model of fractional calculus, which we discussed in Section \ref{sec:analysis}, will be very useful in proving many properties of the tempered fractional model which now follow directly from already known results in the Riemann--Liouville model. We also demonstrated some special functions which have an intrinsic connection to tempered fractional calculus.

Taylor's theorem is a very significant fundamental result in standard calculus, and several analogues of it have been discussed and proved in Riemann--Liouville fractional calculus. It is interesting to see that similar methods of proof still apply in other types of fractional calculus. Finally, we proved some inequalities concerning the tempered fractional integrals, which follow from their definition as an integral transform.

All of these results help to establish a firm foundation for the theory of a new model of fractional calculus. A major direction of current research in fractional calculus is to develop, analyse, and classify new types of operators: to consider their applications to various real-life processes, to study their behaviour in a mathematical sense, and to check how they fit into the overall framework of the field. Our current work contributes to these efforts in the case of one particular model of fractional calculus.



\begin{thebibliography}{99}

\bibitem{anderson-ulness}
D. R. Anderson, D. J. Ulness, ``Newly defined conformable derivatives", \emph{Advances in Dynamical Systems and Applications} 10(2) (2015), pp. 109--137.

\bibitem{atangana-baleanu}
A. Atangana, D. Baleanu, ``New fractional derivatives with nonlocal and non-singular kernel: theory and application to heat transfer model", \emph{Thermal Science} 20(2) (2016), pp. 763--769.

\bibitem{baleanu-diethelm-scalas-trujillo}
D. Baleanu, K. Diethelm, E. Scalas, J. J. Trujillo, \emph{Fractional calculus: models and numerical methods}, 2nd ed., World Scientific, New York, 2017.

\bibitem{baleanu-fernandez}
D. Baleanu, A. Fernandez, ``On some new properties of fractional derivatives with Mittag-Leffler kernel", \emph{Communications in Nonlinear Science and Numerical Simulation} 59 (2018), pp. 444--462.

\bibitem{belarbi-dahmani}
S. Belarbi, Z. Dahmani, ``On some new fractional integral inequalities", \emph{Journal of Inequalities in Pure and Applied Mathematics} 10(3) (2009), Article 86.

\bibitem{buschman}
R. G. Buschman, ``Decomposition of an integral operator by use of Mikusenski calculus", \emph{SIAM Journal of Mathematical Analysis} 3(1) (1972), pp. 83--85.

\bibitem{caputo}
M. Caputo, ``Linear Models of Dissipation whose $Q$ is almost Frequency Independent--II", \emph{Geophysical Journal International} 13(5) (1967), pp. 529--539.

\bibitem{caputo-fabrizio}
M. Caputo, M. Fabrizio, ``A new Definition of Fractional Derivative without Singular Kernel", \emph{Progress in Fractional Differentiation and Applications} 1(2) (2015), pp. 73--85.

\bibitem{cetinkaya-kiymaz-agarwal-agarwal}
A. \c{C}etinkaya, I. O. Kiymaz, P. Agarwal, R. Agarwal, ``A comparative study on generating function relations for generalized hypergeometric functions via generalized fractional operators", \emph{Advances in Difference Equations} 2018:156 (2018).

\bibitem{deng-zhang}
W. Deng, Z. Zhang, \emph{High Accuracy Algorithm for the Differential Equations Governing Anomalous Diffusion}, World Scientific, Singapore, 2019

\bibitem{dugowson}
S. Dugowson, ``Les diff\'erentielles m\'etaphysiques: histoire et philosophie de la g\'en\'eralisation de l'ordre de d\'erivation'', PhD thesis, Universit\'e Paris Nord, 1994.


\bibitem{fernandez-baleanu}
A. Fernandez, D. Baleanu, ``The mean value theorem and Taylor's theorem for fractional derivatives with Mittag-Leffler kernel", \emph{Advances in Difference Equations} 2018:86 (2018).

\bibitem{fernandez-baleanu-srivastava}
A. Fernandez, D. Baleanu, H. M. Srivastava, ``Series representations for models of fractional calculus involving generalised Mittag-Leffler functions", \emph{Communications in Nonlinear Science and Numerical Simulation} 67 (2019), pp. 517-527.

\bibitem{fernandez-ozarslan-baleanu}
A. Fernandez, M. A. \"Ozarslan, D. Baleanu, ``On fractional calculus with general analytic kernels", \emph{Applied Mathematics and Computation} 354 (2019), pp. 248--265.

\bibitem{hilfer}
R. Hilfer, ed., \emph{Applications of Fractional Calculus in Physics}, World Scientific, Singapore, 2000.

\bibitem{hilfer-luchko}
R. Hilfer, Y. Luchko, ``Desiderata for Fractional Derivatives and Integrals'', \emph{Mathematics} 7 (2019), 149.



\bibitem{jarad-abdeljawad-alzabut}
F. Jarad, T. Abdeljawad, J. Alzabut, ``Generalized fractional derivatives generated by a class of local proportional derivatives", \emph{European Physical Journal Special Topics} 226 (2018), pp. 3457--3471.

\bibitem{kilbas-saigo-saxena}
A. A. Kilbas, M. Saigo, R. K. Saxena, ``Generalized Mittag-Leffler function and generalized fractional calculus operators", \emph{Integral Transforms and Special Functions} 15(1) (2004), pp. 31--49.

\bibitem{kiymaz-cetinkaya-agarwal}
I. O. Kiymaz, A. \c{C}etinkaya, P. Agarwal, ``An extension of Caputo fractional derivative operator and its applications'', \emph{Journal of Nonlinear Science and Applications} 9 (2016), pp. 3611--3621.

\bibitem{kobayashi}
K. Kobayashi, ``On generalized gamma functions occurring in diffraction theory'', \emph{Journal of the Physical Society of Japan} 60(5) (1991), pp. 1501--1512.

\bibitem{li-deng-zhao}
C. Li, W. Deng, L. Zhao, ``Well-posedness and numerical algorithm for the tempered fractional ordinary differential equations", \emph{Discrete and Continuous Dynamical Systems - B} 24(4) (2019), pp. 1989--2015.



\bibitem{meerschaert-sabzikar-chen}
M. M. Meerschaert, F. Sabzikar, J. Chen, ``Tempered fractional calculus", \emph{Journal of Computational Physics} 293 (2015), pp. 14--28.

\bibitem{miller-ross}
K. S. Miller, B. Ross, \emph{An Introduction to the Fractional Calculus and Fractional Differential Equations}, Wiley, New York, 1993.

\bibitem{odibat-shawagfeh}
Z. M. Odibat, N. T. Shawagfeh, ``Generalized Taylor's formula'', \emph{Applied Mathematics and Computation} 186 (2007), pp. 286--293.

\bibitem{oldham-spanier}
K. B. Oldham, J. Spanier, \emph{The Fractional Calculus}, Academic Press, San Diego, 1974.

\bibitem{ortigueira-machado}
M. D. Ortigueira, J. A. T. Machado, ``What is a fractional derivative?", \emph{Journal of Computational Physics} 293 (2015), pp. 4--13.

\bibitem{ozarslan-ozergin}
M. A. \"Ozarslan, E. \"Ozergin, ``Some generating relations for extended hypergeometric functions via generalized fractional derivative operator'', \emph{Mathematical and Computer Modelling} 52 (2010), pp. 1825--1833.

\bibitem{ozarslan-ustaoglu1}
M. A. \"Ozarslan, C. Ustao\u{g}lu, ``Some incomplete hypergeometric functions and incomplete Riemann--Liouville fractional integral operators'', \emph{Mathematics} 7(4) (2019), 483.

\bibitem{ozarslan-ustaoglu2}
M. A. \"Ozarslan, C. Ustao\u{g}lu, ``Incomplete Caputo fractional derivative operators'', \textit{Advances in Difference Equations} 2018:209 (2018).

\bibitem{picard}
\'E. Picard, ``Sur une extension aux fonctions de deux variables du probl\`eme de Riemann relatif aux fonctions hyperg\'eom\'etriques" \emph{Annales scientifiques de l'\'Ecole Normale Sup\'erieure} 10 (1881), pp. 305--322.

\bibitem{prabhakar}
T. R. Prabhakar, ``A singular integral equation with a generalized Mittag Leffler function in the kernel", \emph{Yokohama Mathematical Journal} 19 (1971), pp. 7--15.

\bibitem{ross}
B. Ross, ``A Brief History and Exposition of the Fundamental Theory of Fractional Calculus", in B. Ross (ed.), \emph{Fractional Calculus and Its Applications, Lecture Notes in Mathematics No 457}, Springer, Heidelberg, 1975.

\bibitem{samko-kilbas-marichev}
S. G. Samko, A. A. Kilbas, O. I. Marichev, \emph{Fractional Integrals and Derivatives: Theory and Applications}, Taylor \& Francis, London, 2002 [orig. ed. in Russian; Nauka i Tekhnika, Minsk, 1987].


\bibitem{yang-srivastava-machado}
X.-J. Yang, H. M. Srivastava, A. T. Machado, ``A new fractional derivative without singular kernel'', \emph{Thermal Science} 20(2) (2016), pp. 753--756.

\end{thebibliography}
\end{document}